\documentclass[12pt]{amsart}   
\usepackage{amsmath}
\usepackage{amssymb}   
\usepackage{setspace} 
\usepackage{graphicx}	        

\DeclareMathOperator{\Tr}{Tr}

\def\Dic{\text{Dic}} 
\def\tv{\text{TV}}

\newtheorem{theorem}{Theorem}[section]
\newtheorem{proposition}[theorem]{Proposition}
\newtheorem{lemma}[theorem]{Lemma}

\newtheorem*{theorem*}{Theorem}

\title{Random Walks on Dicyclic Group}
\author{Songzi Du \\ Stanford GSB}
\thanks{Contact: songzidu@stanford.edu}
\date{June 18, 2008}

\begin{document}


\maketitle
\markboth{\sc Songzi Du}{\sc Random Walks on Dicyclic Group}


\section{Introduction}
In this paper I work out the rate of convergence of a non-symmetric random walk on the dicyclic group ($\Dic_n$)
and that of its symmetric analogue on the same group.  The analysis is via group representation techniques from Diaconis (1988) and 
closely follows the analysis of random walk on cyclic group in Chapter 3C of Diaconis.  I find that while the mixing times (the time
for the random walk to get ``random'')
for the non-symmetric and the symmetric random walks on $\Dic_n$ are both on the order of $n^2$, 
the symmetric random walk will take approximately twice as long to get ``random'' as the non-symmetric walk.

\section{$\Dic_n$ and its Irreducible Representations}

For integer $n \geq 1$ the dicyclic group (of order $4n$), $\Dic_n$, has the presentation
$$<a, x \mid a^{2n} = 1, x^2 = a^n, x a x^{-1}= a^{-1}>.$$

More concretely, $\Dic_n$ is composed of $4n$ elements: $1, a, a^2, \ldots, a^{2n-1}$, $x, a x, a^2 x, \ldots, a^{2n-1}x$.
The multiplications are as follows (all additions in exponents are modulo $2n$):
\begin{itemize}
\item $(a^k) (a^m) = a^{k+m}$
\item $(a^k) (a^m x) = a^{k+m} x$
\item $(a^k x) (a^m) = a^{k-m} x$
\item $(a^k x) (a^m x) = a^{k-m+n}$
\end{itemize}

$\Dic_2$ is the celebrity among all $\Dic_n$'s and is known as the quaternion group.

It's clear that $\{1, a, a^2, \ldots, a^{2n-1}\}$ is an Abelian subgroup of $\Dic_n$ of index 2.  Therefore,
each irreducible representation of $\Dic_n$ has dimension less than or equal to $2$ (Corollary to Theorem 9 in Serre (1977)).

When $n$ is odd, the 1-D irreducible representations of $\Dic_n$ are:
\begin{itemize}
\item $\psi_0(a)=1, \psi_0(x)=1$
\item $\psi_1(a)=1, \psi_1(x)=-1$
\item $\psi_2(a)=-1, \psi_2(x)=i$
\item $\psi_3(a)=-1, \psi_3(x)=-i$
\end{itemize}

When $n$ is even, the 1-D irreducible representations of $\Dic_n$ become:
\begin{itemize}
\item $\psi_0(a)=1, \psi_0(x)=1$
\item $\psi_1(a)=1, \psi_1(x)=-1$
\item $\psi_2(a)=-1, \psi_2(x)=1$
\item $\psi_3(a)=-1, \psi_3(x)=-1$
\end{itemize}

One can work out 1-D representation $\psi$ using the identity $\psi(a x a)=\psi(x)$, 
and thus $\psi(a)=\pm 1$.  If $n$ is even, this means $\psi(x^2) = \psi(a^n) = 1$, thus
$\psi(x) = \pm 1$ as well.  When $n$ is odd, we can have $\psi(a^n) = -1$, so $\psi(x)$ can be
$\pm i$.  This explains the difference in 1-D representations when $n$ is even or odd.

The 2-D irreducible representations of $\Dic_n$ are:

\begin{equation}
\label{2drep}
\rho_r(a)=\begin{pmatrix} \omega^r & 0  \\ 0 & \omega^{-r} \end{pmatrix}, 
\rho_r(x)=\begin{pmatrix} 0 & (-1)^r  \\ 1 & 0 \end{pmatrix},
\end{equation}

where $1 \leq r \leq n-1$ and $\omega = e^{\pi i/n}$.  One can easily check
that these representations are all distinct and irreducible (by checking their characters), and
that indeed we have $\rho_r(x a x^{-1}) = \rho_r(a^{-1})$, $\rho_r(x^2)=\rho_r(a^n)$, and
$\rho_r(a^{2n})=\rho_r(1)$.

Clearly, all $\psi_0, \ldots, \psi_3, \rho_1, \ldots, \rho_{n-1}$ are unitary representations
(conjugate transposes being their inverses).  They are all of the possible irreducible representations
of $\Dic_n$, because $4 + 4(n-1) = 4n$, which is the order of $\Dic_n$.

\section{An Asymmetric Random Walk on $\Dic_n$}
In this and next section, we will assume that $n$ is odd.  

Let $Q(a) = Q(x) = 1/2$ ($Q(s) = 0$ for other $s \in \Dic_n$), and 
consider the non-symmetric random walk generated on $\Dic_n$ by $Q$: the probability of going from
$s$ to $t$ is $Q(ts^{-1})$, $s, t \in \Dic_n$.  Clearly, the uniform distribution on $\Dic_n$, $U(s)=1/4n$, $s \in \Dic_n$, is
a stationary distribution of random walk $Q$.  In fact, it is the only one, as $Q$ is irreducible
and aperiodic: $Q$ being irreducible is clear; aperiodicity of $Q$ follows from identities $a^{2n}=1$ and
$a^{n+1} x a x = 1$, and the following lemma:

\begin{lemma}
If $n$ is odd, the greatest common divisor of $2n$ and $n+4$ is 1.
\end{lemma}
\begin{proof}
Suppose integer $k \geq 1$ divides both $2n$ and $n+4$.  

If $k$ divides $n$,
then this together with $k$ dividing $n+4$ implies that $k$ divides $4$ as well, i.e.\ 
$k = 1, 2,$ or $4$; since $n$ is odd, this means that $k=1$.

If $k$ does not divide $n$, then $k$ dividing $2n$ implies that 2 divides $k$; but 
this together with $k$ dividing $n+4$ implies that 2 divides $n$; thus we can't have
$k$ not dividing $n$.
\end{proof}

We are interested in how ``well mixed'' (or ``random'') the random walk $Q$ is after $k$ steps, i.e.\ the total variation distance to the uniform distribution $U(s) = 1/4n$, $s \in \Dic_n$:
\begin{equation}
\label{tv}
\| Q^{*k} - U \|_\tv = \max_{A \subseteq \Dic_n} | Q^{*k}(A) - U(A) |
\end{equation}
where $Q^{*k}$ is the $k$-th convolution of $Q$ with itself: $Q^{*1}=Q$, and 
$Q^{*k}(s)= \sum_{t \in \Dic_n} Q(s t^{-1}) Q^{*k-1}(t)$ for $s \in \Dic_n$ and $k \geq 2$.

\subsection{Upper Bound}
Our upper bound on the distance to stationarity (Equation \ref{tv}) comes from Lemma 1
of Diaconis (1988), Chapter 3B:
\begin{lemma}
For any probability measure $P$ on a finite group $G$, 
$$\| P - U \|_{\tv}^2 \leq \frac{1}{4} \sum_{\rho \neq 1} d_\rho \Tr(\hat{P}(\rho) \hat{P}(\rho)^{*}),$$
where $U(s) = 1/|G|$, the summation is over all non-trivial unitary (* here refers to conjugate transpose)
irreducible representations of $G$, $d_\rho$ is the degree of the representation $\rho$, 
and $\hat{P}(\rho) = \sum_{s \in G} P(s) \rho(s)$
is the Fourier transform of $P$ at the representation $\rho$.
\end{lemma}

Specializing to $G = \Dic_n$ and $P=Q^{*k}$, we get that
\begin{equation}
\label{upperbound}
\| Q^{*k} - U \|_{\tv}^2 \leq \frac{1}{4} \left( \sum_{i=1}^{3} \hat{Q}(\psi_i)^k (\hat{Q}(\psi_i)^k)^* +  
\sum_{r=1}^{n-1}2 \Tr \left( \hat{Q}(\rho_r)^k (\hat{Q}(\rho_r)^k)^* \right) \right),
\end{equation}
where $\psi_i$ and $\rho_r$ are listed in the previous section.

For $\psi_i$, we find that $\hat{Q}(\psi_1)^k (\hat{Q}(\psi_1)^k)^* = 0$ and 
$\hat{Q}(\psi_2)^k (\hat{Q}(\psi_2)^k)^* = \hat{Q}(\psi_3)^k (\hat{Q}(\psi_3)^k)^* = 2^{-k}$.

For odd $r$ such that $1 \leq r \leq n-1$, we have
$$\hat{Q}(\rho_r) \hat{Q}(\rho_r)^* = \frac{1}{4} 
\begin{pmatrix} 2 & 0  \\ 0 & 2   \end{pmatrix}. $$ 

As a result,
$$\Tr \left( \hat{Q}(\rho_r)^k (\hat{Q}(\rho_r)^k)^{*} \right) = 2^{-k+1},$$
holds for $r$ odd, $1 \leq r \leq n-1$.

For even $r$, we use the diagonalization (recall that $\omega = e^{\pi i /n}$)
\begin{equation}
\label{diag1}
2 \hat{Q}(\rho_r) = \begin{pmatrix}\omega^r & 1 \\ 1 & \omega^{-r} \end{pmatrix} = 
\begin{pmatrix} \omega^{-r} & \omega^r \\ -1 & 1 \end{pmatrix} 
\begin{pmatrix} 0 & 0 \\ 0 & \omega^r+\omega^{-r} \end{pmatrix}
\left(\frac{1}{\omega^r+\omega^{-r}} \begin{pmatrix} 1 & -\omega^r \\ 1 & \omega^{-r} \end{pmatrix} \right),
\end{equation}
so that
$$ \Tr\left( 2^k \hat{Q}(\rho_r)^k (2^k\hat{Q}(\rho_r)^k)^* \right) = 4 (\omega^r+\omega^r)^{2k-2},$$
and therefore, for $r$ even, $1 \leq r \leq n-1$, 
$$\Tr\left( \hat{Q}(\rho_r)^k (\hat{Q}(\rho_r)^k)^{*} \right) = \left( \cos \frac{r \pi}{n} \right)^{2k-2}.$$

Plug these into (\ref{upperbound}), we have:
$$
\| Q^{*k} - U \|_{\tv}^2 \leq n 2^{-k-1} + 
\frac{1}{2} \sum_{\substack{r=2 \\r \text{ even}}}^{n-1} \left( \cos \frac{r \pi}{n} \right)^{2k-2}.
$$

We need to bound the second term on the right hand side.  First notice that
$$ \sum_{\substack{r=2 \\r \text{ even}}}^{n-1} \left( \cos \frac{r \pi}{n} \right)^{2k-2}
= \sum_{r=1}^{(n-1)/2} \left( \cos \frac{r \pi}{n} \right)^{2k-2}.
$$

We then can use $\cos(x) \leq e^{-x^2/2}$, for $x \in [0, \pi/2]$, (see Appendix for derivations of this and
other cosine inequalities used in the paper) to get, for $k \geq 2$:
$$ \sum_{r=1}^{(n-1)/2} \cos \left( \frac{r \pi}{n} \right)^{2k-2} \leq 
\sum_{r=1}^{(n-1)/2} \exp\left( - \frac{r^2 \pi^2}{n^2} (k-1) \right)
\leq \frac{\exp(-\pi^2(k-1)/n^2)}{1-\exp(-3 \pi^2(k-1)/n^2)}.
$$

Therefore, our upper bound is, for odd $n$ and $k \geq 2$,
\begin{equation}
\| Q^{*k} - U \|_{\tv}^2 \leq n 2^{-k-1} + \frac{\exp(-\pi^2(k-1)/n^2)}{2(1-\exp(-3 \pi^2(k-1)/n^2))}.
\end{equation}

\subsection{Lower Bound}

One can easily show that for any two probability measures $P_1$ and $P_2$ on a finite set $X$,
$$ \| P_1 - P_2 \|_\tv = \frac{1}{2} \max_{\substack{f: X \rightarrow \mathbb{R} \\ \|f\|\leq 1}} | P_1(f) - P_2(f) |,
$$
where $P_1(f) = \sum_{x \in X} f(x) P_1(x)$, and likewise for $P_2(f)$.

For any finite group $G$, the uniform distribution on $G$, $U(s) = 1/|G|$, enjoys the property
that $\hat{U}(\rho) = 0$ for any non-trivial irreducible representation $\rho$ of $G$ (Exercise 3 of Chapter 2B
in Diaconis (1988)).

Let $f(s) = \frac{1}{2} \Tr\rho_r(s)$ for $s \in \Dic_n$, where $1 \leq r \leq n-1$ and
$\rho_r$ is as in Equation~\ref{2drep}, we check that $f(s)$ is a real number and that $|f(s)| \leq 1$, for all $s \in \Dic_n$;
and $U(f) = \Tr \hat{U}(\rho_r)/2 = 0$.  Therefore, we have
\begin{equation}
\label{lowerbound}
\| Q^{*k} - U \|_\tv \geq \frac{1}{2} \left| \sum_{s \in \Dic_n}Q^{*k}(s) f(s)\right|
= \frac{1}{4} \left| \Tr \left( \hat{Q}(\rho_r)^k \right) \right|.
\end{equation}

Using the diagolization in Equation~\ref{diag1}, we have for even $r$ such that
$1 \leq r \leq n-1$,
\begin{equation*}
\| Q^{*k} - U \|_\tv \geq \frac{1}{4} \left| \cos \left(\frac{r\pi}{n}\right)^k \right|.
\end{equation*}
We can let $r=n-1$, and get for $n \geq 7$,
\begin{equation*}
\| Q^{*k} - U \|_\tv \geq \frac{1}{4} \cos \left( \frac{\pi}{n} \right)^k \geq 
\frac{1}{4}\exp \left( -\frac{\pi^2 k}{2n^2} - \frac{\pi^4 k}{12 n^4} - \frac{17 \pi^5 k}{120 n^5} \right),
\end{equation*}
by the inequality $e^{-x^2/2-x^4/12-17x^5/120} \leq \cos(x)$ for $0\leq x \leq 1/2$.

We summarize the results of this section in the following theorem:

\begin{theorem}
\label{thm:nonsym}
Fix the probability measure $Q$ on $\Dic_n$ such that $Q(a)=Q(x)=1/2$.

For any odd $n \geq 1$ and any $k\geq 2$, we have
$$\| Q^{*k} - U \|_{\tv}^2 \leq n 2^{-k-1} + \frac{\exp(-\pi^2(k-1)/n^2)}{2(1-\exp(-3 \pi^2(k-1)/n^2))}.$$

For any odd $n \geq 7$ and any $k\geq 1$, we have
$$\frac{1}{4}\exp \left( -\frac{\pi^2 k}{2n^2} - \frac{\pi^4 k}{12 n^4} - \frac{17 \pi^5 k}{120 n^5} \right) \leq \| Q^{*k} - U \|_{\tv}.$$
\end{theorem}

\section{A Symmetric Random Walk on $\Dic_n$}
We now consider the symmetrization of the previous random walk:
let $Q$ be such that $Q(a)=Q(a^{2n-1})=Q(x)=Q(a^nx)=1/4$ ($Q(s) = 0$ for other $s \in \Dic_n$);
and $Q(ts^{-1})$ is still the probability of going from $s$ to $t$, $s, t \in \Dic_n$.  
As before, we assume that $n$ is odd. Clearly, the uniform distribution 
$U(s)=1/4n$, $s\in\Dic_n$, is still the unique stationary distribution of this new random walk $Q$.
And as before, we are interested in bounding Equation~\ref{tv}.

\subsection{Upper Bound}
We first note that Inequality~\ref{upperbound} still holds for our new $Q$.

Now, we have $\hat{Q}(\psi_1)^k (\hat{Q}(\psi_1)^k)^*=0$, and 
$\hat{Q}(\psi_2)^k (\hat{Q}(\psi_2)^k)^* = \hat{Q}(\psi_3)^k (\hat{Q}(\psi_3)^k)^* = 1/4^k$.

For $1 \leq r \leq n-1$, 
\begin{equation*}
\hat{Q}(\rho_r) = \frac{1}{2} \begin{pmatrix} \cos \frac{r\pi}{n} & \frac{1}{2}(-1)^r + \frac{1}{2}  
\\ \frac{1}{2}(-1)^r + \frac{1}{2}  &   \cos \frac{r\pi}{n}  \end{pmatrix}.
\end{equation*}

Thus, for odd $r$, $1 \leq r \leq n-1$, 
\begin{equation*}
\Tr \left( \hat{Q}(\rho_r)^k (\hat{Q}(\rho_r)^k)^* \right )
= \frac{2}{4^k} \left( \cos \frac{r \pi}{n} \right)^{2k}.
\end{equation*}

For even $r$, $1 \leq r \leq n-1$, we have the diagonalization
\begin{equation}
\label{diag2}
2 \hat{Q}(\rho_r) = 
\begin{pmatrix} \frac{1}{\sqrt{2}} & \frac{1}{\sqrt{2}} \\ \frac{1}{\sqrt{2}} & -\frac{1}{\sqrt{2}} \end{pmatrix}
\begin{pmatrix} \cos \frac{r\pi}{n} + 1 & 0 \\ 0 & \cos \frac{r\pi}{n} - 1 \end{pmatrix}
\begin{pmatrix} \frac{1}{\sqrt{2}} & \frac{1}{\sqrt{2}} \\ \frac{1}{\sqrt{2}} & -\frac{1}{\sqrt{2}} \end{pmatrix}.
\end{equation}

Thus, for even $r$, $1 \leq r \leq n-1$,
\begin{equation*}
\Tr \left( \hat{Q}(\rho_r)^k (\hat{Q}(\rho_r)^k)^* \right )
= \frac{1}{4^k} \left( \left( \cos \frac{r \pi}{n} + 1 \right)^{2k} + \left( \cos \frac{r \pi}{n} - 1 \right)^{2k} \right).
\end{equation*}

Inequality~\ref{upperbound} therefore translates to
\begin{align*}
\| Q^{*k} - U \|_{\tv}^2 & \leq \frac{1}{2} \left( \frac{1}{4^k} \right) + 
\sum_{\substack{r=1\\r \text{ odd}}}^{n-1} \left( \frac{1}{2} \cos\frac{r \pi}{n} \right)^{2k} \\
& + \frac{1}{2} \sum_{\substack{r=2\\r \text{ even}}}^{n-1} \frac{1}{4^k} \left( \left(\cos \frac{r \pi}{n} + 1\right)^{2k}
+ \left( \cos \frac{r \pi}{n} -1 \right)^{2k} \right).
\end{align*}

The second term in the right hand side is easily bounded:
$$\sum_{\substack{r=1\\r \text{ odd}}}^{n-1} \left( \frac{1}{2} \cos\frac{r \pi}{n} \right)^{2k} \leq \frac{n-1}{2} \frac{1}{4^{k}}.$$

And for the third term,
\begin{align*}
& \sum_{\substack{r=2\\r \text{ even}}}^{n-1} \frac{1}{4^k} \left( \left(\cos \frac{r \pi}{n} + 1\right)^{2k}
+ \left( \cos \frac{r \pi}{n} -1 \right)^{2k} \right) \\
& \leq \sum_{r=1}^{(n-1)/2} \frac{1}{4^k} \left( 1+ \cos\frac{r \pi}{n} \right)^{2k} + \frac{n-1}{2} \frac{1}{4^k} \\
& \leq \sum_{r=1}^{(n-1)/2} \exp(-r^2 \pi^2 k / 2 n^2) + \frac{n-1}{2} \frac{1}{4^k} \\
& \leq \frac{\exp(-\pi^2 k / 2n^2)}{1-\exp(-3\pi^2 k / 2n^2)} + \frac{n-1}{2} \frac{1}{4^k},
\end{align*}
where we used the inequality $(1+\cos x)/2 \leq e^{-x^2/4}$ for all $x\in [0, \pi]$.

Collecting the terms together, we have, 
\begin{equation*}
\| Q^{*k} - U \|_\tv^2 \leq \frac{3n-1}{4} \frac{1}{4^k} + \frac{\exp(-\pi^2 k / 2n^2)}{2 (1-\exp(-3\pi^2 k / 2n^2))}.
\end{equation*}
\subsection{Lower Bound}
Inequality~\ref{lowerbound} is valid for our new $Q$ as well.  Using the diagonalization in (\ref{diag2}) for $\rho_{n-1}$ 
(assuming that $n>1$), Inequality~\ref{lowerbound} implies that
\begin{align*}
\| Q^{*k} - U \|_\tv & \geq \frac{1}{4} \left( \frac{1}{2^k} \right)
\left| \left( -\cos \frac{\pi}{n}+1 \right)^k + \left( -\cos \frac{\pi}{n}-1 \right)^k \right| \\
& \geq \frac{1}{4} \left( \left( \frac{1+\cos(\pi/n)}{2} \right)^k - \frac{1}{2^k} \right) \\
& \geq \frac{1}{4} \left( \exp \left( -\frac{\pi^2 k}{4n^2} - \frac{\pi^4 k} {96 n^4} - \frac{\pi^5 k}{400n^5} \right) - \frac{1}{2^k} \right)
\end{align*}
for $n \geq 7$; in the last time we used the inequality $e^{-x^2/4-x^4/96-x^5/400} \leq (\cos x+1)/2$ for $x\in[0, 1/2]$.

Therefore, we arrive at
\begin{theorem}
\label{thm:sym}
Fix the probability measure $Q$ on $\Dic_n$ such that $Q(a)=Q(a^{2n-1})=Q(x)=Q(a^n x)=1/4$.

For any odd $n \geq 1$ and any $k\geq 1$, we have
$$\| Q^{*k} - U \|_\tv^2 \leq \frac{3n-1}{4} \frac{1}{4^k} + \frac{\exp(-\pi^2 k/2n^2)}{2 (1-\exp(-3\pi^2 k / 2n^2))}$$.

For any odd $n \geq 7$ and any $k\geq 1$, we have
$$  \frac{1}{4} \exp \left( -\frac{\pi^2 k}{4n^2} - \frac{\pi^4 k} {96 n^4} - \frac{\pi^5 k}{400n^5} \right) - \frac{1}{2^{k+2}}
\leq \| Q^{*k} - U \|_{\tv}.$$
\end{theorem}
Comparing the bounds in Theorem~\ref{thm:nonsym} to that in Theorem~\ref{thm:sym}, we conclude that the mixing time
for the non-symmetric random walk is approximately half of the mixing time of the symmetric random walk.  It would be
interesting to give a ``purely'' probabilistic proof of this phenomenon.

\appendix

\section{Some Inqualities on Cosine}

\begin{proposition}
For $x \in [0, \pi/2]$, $\cos x \leq e^{-x^2/2}$.
\end{proposition}
\begin{proof}
We will show that for $x \in [0, \pi/2]$, $\log(\cos x ) \leq -x^2/2$.  Clearly this is true when $x=0$.
And we have
\begin{equation*}
\frac{d \log(\cos x)}{dx} = - \tan(x) \leq -x
\end{equation*}
 because $\frac{d \tan x}{dx} = 1/\cos^2(x) \geq 1$ for $x \in [0, \pi/2]$.
\end{proof}

\begin{proposition}
For $x \in [0, \pi]$, $(1+\cos x)/2 \leq e^{-x^2/4}$.
\end{proposition}
\begin{proof}
We will show that for $x \in [0, \pi]$, $\log((1+\cos x)/2) \leq -x^2/4$.  Clearly it holds for $x = 0$.
And we have
\begin{equation*}
\frac{d \log \frac{1+\cos x}{2} }{dx} = -\frac{\sin(x)}{(1+\cos(x)} \leq -\frac{x}{2}
\end{equation*}
because 
\begin{equation*}
\frac{d \frac{\sin(x)}{1+\cos(x)} }{dx} = \frac{1}{1+\cos(x)} \geq \frac{1}{2}
\end{equation*}
 for $x \in [0, \pi]$.
\end{proof}

\begin{proposition}
For $x \in [0, 1/2]$, $e^{-x^2/2-x^4/12-17x^5/120} \leq \cos(x)$.
\end{proposition}
\begin{proof}
Taylor expansion of $\log(\cos x)$ around 0 gives for any $0 < x \leq 1/2$:
\begin{equation*}
\log (\cos x) = - \frac{x^2}{2} - \frac{x^4}{12} + f(x') \frac{x^5}{120},
\end{equation*}
where $0 < x' < x$ and 
$$f(x') = -16 \tan(x')\sec(x')^4-8\tan(x')^3 \sec(x')^2$$ is the
fifth derivative of $\log(\cos x)$.  Clearly, $|f(x')| \leq |f(1/2)| \leq 17$; thus
\begin{equation*}
\log (\cos x) \geq - \frac{x^2}{2} - \frac{x^4}{12} - \frac{17}{120} x^5
\end{equation*}
\end{proof}

\begin{proposition}
For $x \in [0, 1/2]$, $e^{-x^2/4-x^4/96-x^5/400} \leq (\cos x+1)/2$.
\end{proposition}
\begin{proof}
Taylor expansion of $\log ( (\cos x+1)/2)$ around 0 gives for any $0 < x \leq 1/2$:
\begin{equation*}
\log ( (\cos x+1)/2) = - \frac{x^2}{4} - \frac{x^4}{96} + f(x') \frac{x^5}{120},
\end{equation*}
where $0 < x' < x$ and 
$$f(x') = \frac{1}{8} \sec(x'/2)^5 (-11 \sin(x'/2) + \sin(3x'/2))$$ is the
fifth derivative of $\log ( (\cos x+1)/2)$.  

Since the derivative of $f$,
$$f'(z) = -\frac{1}{16}(33 - 26 \cos(z) + \cos(2z)) \sec(z/2)^6 < 0$$
for all $0 \leq z \leq 1/2$, we conclude that $f(z)<0$ and 
$|f(x')| \leq |f(1/2)| \leq 0.3$.

Therefore,
\begin{equation*}
\log ( (\cos x+1)/2) \geq - \frac{x^2}{2} - \frac{x^4}{96} - \frac{0.3}{120} x^5.
\end{equation*}
\end{proof}

\end{document}